\documentclass[11pt]{article}

\usepackage{amsmath,amsfonts,amsthm,amssymb,graphicx}
\usepackage[all,2cell,ps]{xy}
\usepackage{authblk}
\usepackage[pdfborder={0 0 0}]{hyperref}
\usepackage{color}

\theoremstyle{plain}
\newtheorem{Theorem}{Theorem}[section]

\newtheorem{Algorithm}[Theorem]{Algorithm}
\newtheorem*{CSP}{Congruence Subgroup Problem}

\theoremstyle{definition}
\newtheorem{Remark}[Theorem]{Remark}

\newcommand{\Q}{\mathbb{Q}}
\newcommand{\R}{\mathbb{R}}
\newcommand{\Z}{\mathbb{Z}}
\newcommand{\Qp}{\mathbb{Q}_p}
\newcommand{\Zp}{\mathbb{Z}_p}

\newcommand{\ZS}{\mathbb{Z}_S}
\newcommand{\F}{\mathbb{F}}
\newcommand{\Fp}{\mathbb{F}_p}

\DeclareMathOperator{\M}{\mathrm{M}}
\DeclareMathOperator{\SL}{\mathrm{SL}}
\DeclareMathOperator{\PSL}{\mathrm{PSL}}
\DeclareMathOperator{\PGL}{\mathrm{PGL}}
\DeclareMathOperator{\GL}{\mathrm{GL}}
\DeclareMathOperator{\Nm}{\mathrm{N}}

\newcommand{\HamQ}{\mathcal{H}}
\newcommand{\HamR}{\mathbb{H}}
\newcommand{\Hamp}{\mathcal{H}_p}
\newcommand{\Hur}{\mathcal{O}}
\newcommand{\HurS}{\mathcal{O}_S}
\newcommand{\Sunits}{\Gamma_S}
\newcommand{\PSunits}{\overline{\Gamma}_S}
\newcommand{\ProjHur}{\overline{\mathcal{O}}^*}

\newcommand{\Xp}{\mathfrak{X}_p}

\newcommand{\XS}{\mathfrak{X}_S}

\newcommand{\Rels}{\mathcal{R}}

\DeclareMathOperator{\Afour}{\mathrm{A}_4}

\title{Presentations for quaternionic $S$-unit groups}

\author{Ted Chinburg}
\affil{University of Pennsylvania \\ \url{ted@math.upenn.edu} }
 
\author{Holley Friedlander}
\affil{Williams College \\ \url{hf2@williams.edu}}

\author{Sean Howe}
\affil{University of Chicago \\ \url{seanpkh@gmail.com}}

\author{Michiel Kosters}
\affil{Universiteit Leiden \\ \url{mkosters@math.leidenuniv.nl}}

\author{Bhairav Singh}
\affil{ MIT \\ \url{bsingh@mit.edu}}

\author{Matthew Stover}
\affil{ Temple University\\ \url{mstover@temple.edu}}

\author{Ying Zhang}
\affil{University of Iowa\\ \url{ying-zhang-1@uiowa.edu}}

\author{Paul Ziegler}
\affil{ETH Zurich \\ \url{paul.ziegler@math.ethz.ch}}

\date{\today}

\begin{document}

\maketitle

%%%%%%%%%%%%%%%%%%%%
\section{Introduction and notation}\label{sec:intro}
%%%%%%%%%%%%%%%%%%%%

The purpose of this paper is to give presentations for projective $S$-unit groups of the Hurwitz order in Hamilton's quaternions over the rational field $\Q$. To our knowledge, this provides the first explicit presentations of an $S$-arithmetic lattice in a semisimple Lie group with $S$ large. In particular, we give presentations for groups acting irreducibly and cocompactly on a product of Bruhat--Tits trees. We also include some discussion and experimentation related to the congruence subgroup problem, which is open when $S$ contains at least two odd primes. In the appendix, we provide code that allows the reader to compute presentations for an arbitrary finite set $S$.

We now introduce the objects studied in this paper, assuming some familiarity with the theory of quaternion algebras over number fields, e.g., from \cite{Vigneras}. Throughout this paper, $\HamQ$ will denote the rational quaternion algebra with basis $\{1, I, J, IJ\}$ subject to the relations
\[
I^2 = J^2 = -1 \quad \quad I J = - J I.
\]
Tensoring over the real numbers, we have $\HamQ \otimes_\Q \R \cong \HamR$, where $\HamR$ denotes Hamilton's quaternions. Let $\Hur$ denote the \emph{Hurwitz order} of $\HamQ$, i.e., the maximal $\Z$-order with basis
\[
\left\{ 1,\ I,\ J,\ \frac{1}{2}(1 + I + J + I J) \right\}.
\]
For each prime $p$, let $\Hamp = \HamQ \otimes_\Q \Qp$ be the completion of $\HamQ$ at $p$. If $p$ is odd, we fix an isomorphism $\Hamp \cong \M_2(\Qp)$ that identifies $\Hur \otimes_\Z \Z_p$ with $\M_2(\Z_p)$. The algebra $\HamQ_2$ is isomorphic to the unique quaternion division algebra over $\Q_2$ with unique maximal $\Z_2$-order $\Hur \otimes_\Z \Z_2$. Let $\Nm$ denote the reduced norm on $\HamQ$, which is given by
\[
\Nm(a + b I + c J + d I J) = a^2 + b^2 + c^2 + d^2.
\]

For any finite, possibly empty, set of rational primes $S$, set $\Z_\emptyset = \Z$ and for $S = \{p_1, \dots, p_r\}$ the $S$-\emph{integers} are
\[
\ZS = \Z \left[ \frac{1}{p_1}, \dots, \frac{1}{p_r} \right].
\]
Also, set
\[
m_S = \prod_{p \in S} p.
\]
Let $\HurS$ be the $S$-\emph{order} $\Hur \otimes_\Z \ZS$. For any such $S$, $\Sunits$ will denote the $S$-\emph{unit group} of invertible elements in the ring $\HurS$. It is well-known that $\Gamma_\emptyset = \Hur^*$ is the binary tetrahedral group \cite[\S V.3. A]{Vigneras}, which also isomorphic to the group $\SL_2(\F_3)$.

To find presentations for $\Sunits$ (more precisely, the group $\PSunits$ of $S$-units modulo scalars), we study the discrete action of $\Sunits$ on a product of \emph{Bruhat--Tits trees}. For any odd prime $p$, let $\Xp$ be the Bruhat--Tits tree associated with $\PGL_2(\Qp)$. See \cite{Serre} for the construction of $\Xp$. For $p = 2$ let $\Xp$ be a single point with trivial action by $\HamQ_2^*$. For $S$ as above, define
\[
\XS = \prod_{p \in S} \Xp.
\]
Under the homomorphism 
\[
\alpha_S : \Sunits \to \prod_{p \in S} \Hamp^* 
\]
we obtain an action of $\Sunits$ on $\XS$ via our chosen isomorphism between $\Hamp^*$ and $\mathrm{GL}_2(\Q_p)$ for odd $p$ and the natural action of $\PGL_2(\Qp)$ on $\Xp$. It is well know that this action is discrete and cocompact, i.e., the natural simplicial structure on $\XS$ coming from the trees $\Xp$ makes $\Sunits \backslash \XS$ a finite CW complex. (This is also clear from Theorem \ref{thm:VertexTransitive} below.) Note that the scalars in $\Sunits$ act trivially, so the $\Sunits$ action factors through the projection onto $\PSunits$.

%%%%%%%%%%%%%%%%%%%%
\subsubsection*{Acknowledgments} We thank the organizers of the 2012 Arizona Winter School, where this work began. This material is based upon work supported by the National Science Foundation under Grant Numbers NSF 0943832 and 1361000 and the National Science Foundation Graduate Research Fellowship Program under Grant No.\ DGE-1144082.
%%%%%%%%%%%%%%%%%%%%

%%%%%%%%%%%%%%%%%%%%
\section{A fundamental domain and the basic algorithm}\label{sec:methods}
%%%%%%%%%%%%%%%%%%%%

Our presentation for $\PSunits$ depends on finding a nice fundamental domain for its action on $\XS$. By a \emph{vertex} of $\XS$, we will always mean a product of vertices of the factors. If $x$ is a vertex of $\Xp$, let $N(x)$ be the set of neighboring vertices. If $v = (v_p)_{p \in S}$ is a vertex of $\XS$, let $N(v) = \prod N(v_p)$.

For each odd $p \in S$, let $v_{0, p}$ be the vertex of the tree $\Xp$ associated with the standard lattice
\begin{equation}
\label{eq:standardL}
\mathcal{L}_{0, p} = \Zp \oplus \Zp.
\end{equation}
Under the action of $\Hamp^* = \GL_2(\Q_p)$ on $\Xp$, $v_{0, p}$ is stabilized by
\[
(\Hur \otimes_\Z \Z_p)^* \cdot \Q_p^* = \mathrm{GL}_2(\Z_p) \cdot \Q_p^*.
\]
If $2$ is in $S$, let $v_{0, p}$ be the unique vertex of $\Xp$. Our fundamental domain is an open neighborhood of the vertex $v_0 = (v_{0, p})_{p \in S}$ in $\XS$.

Note that the stabilizer of $v_0$ in $\Sunits$ is $\Hur^* \cdot \Z_S$. We will apply the following result, which two of the authors proved as an application of the main result of their joint paper \cite{Chinburg--Stover}. For completeness we give an alternate proof of a more classical flavor that replaces \cite{Chinburg--Stover} with an application of Jacobi's theorem on sums of four squares.

%%%%%%%%%%%%%%%%%%%%
\begin{Theorem}\label{thm:VertexTransitive}
The action of $\Sunits$ (hence of $\PSunits$) on $\XS$ is vertex transitive.
\end{Theorem}
%%%%%%%%%%%%%%%%%%%%

%%%%%%%%%%%%%%%%%%%%
\begin{proof}
Let $d_{\Xp}$ be the distance function on $\Xp$ under which adjacent vertices are distance one and $d$ the distance on $\XS$ defined by summing the distances on components. If $\Sunits$ does not act transitively on $\XS$, then there is a vertex $v$ with minimal positive distance from $v_0$ that is not in the orbit $\Sunits \cdot v_0$. We claim that $d(v_0,v) = 1$. If not, then there is a vertex $v'$ with $d(v_0,v') < d(v_0,v)$ and $d(v',v) < d(v_0,v)$, since the $\Xp$ are trees. However, then there is a $\sigma \in \Sunits$ such that $\sigma v_0 = v'$ and $d(v_0,\sigma^{-1} v) = d(v',v) < d(v_0,v)$, so there is a $\tau \in \Sunits$ with $\tau v_0 = \sigma^{-1} v$. This would give $\sigma \tau v_0 = v$, contradicting $v \not \in \Sunits \cdot v_0$, so in fact $d(v_0,v) = 1$.

We conclude that $v = (v_p)_{p \in S}$ has the property that there is a unique odd $p \in S$ such that $v_p \ne v_{0, p}$. Consider the set $N(v_{0, p})$ of vertices in $\Xp$ adjacent to $v_{0, p}$. To prove the theorem, it suffices to show that $N(v_{0, p}) = T(p) \cdot v_{0, p}$ for each $p \in S \setminus \{2\}$, where $T(p) \subset \Sunits$ is the set of elements of $\Hur$ with reduced norm $p$. Indeed, this proves that $\PSunits$ sends $v_0$ to any vertex of $\XS$ distance one from $v_0$ which, by the previous paragraph and the fact that $T(p) \cdot v_{0, q} = v_{0, q}$ for $q \neq p$, implies that $\PSunits$ must act transitively on the vertices of $\XS$.

It is easy to see that indeed $T(p) \cdot v_{0, p} \subseteq N(v_{0, p})$. If $\sigma, \tau \in T(p)$ with $\tau v_0 = \sigma v_0$, then $\tau^{-1} \sigma$ stabilizes $v_0$ and has reduced norm $1$, so it lies in $\Hur^*$. Thus, to prove that $T(p) \cdot v_{0, p} = N(v_{0, p})$ it suffices to show that there are exactly $p + 1 = \# N(v_{0, p})$ cosets of $\Hur^*$ in $T(p)$. This is a consequence of $\# \Hur^* = 24$ together with Jacobi's theorem on sums of four squares. More precisely, there are exactly $24 (p + 1)$ elements of odd prime reduced norm $p$ in $\Hur$. To see this, note that the equation $(a/2)^2 + (b/2)^2 + (c/2)^2 + (d/2)^2 = p$ for $a, b, c, d \in \Z$ implies that they are either all even or all odd and thus
\[
\frac{1}{2}(a + b I + c J + d I J) \in \Hur.
\]
Now multiply by 4 and apply the even case of Jacobi's theorem.
\end{proof}
%%%%%%%%%%%%%%%%%%%%

%%%%%%%%%%%%%%%%%%%%
\begin{Remark}\label{rem:NeighborTransitive}
In fact, the proof shows something stronger, which we will need in Algorithms \ref{alg:okay} and \ref{alg:better} below. For any element $\overline{g} \in \PSunits$ mapping $v_0$ to $v' = \{v'_p\}_{p \in S}$ with $d_{\Xp}(v_{0,p},v'_p)\leq 1$ for each $p \in S$, there exists an element $x \in \Hur$ of reduced norm dividing $m_S := \prod_{p\in S} p$ such that $\bar{x}=\bar{g}$.
\end{Remark}
%%%%%%%%%%%%%%%%%%%%

Applying Theorem \ref{thm:VertexTransitive}, we obtain a fundamental domain
\begin{equation}
\label{eq:Rdomain}
R_S = \prod_{p \in S} R_p
\end{equation}
for the action of $\Sunits$ on $\XS$, where
\begin{equation}\label{eq:Rpdomain}
R_p = \left\{ (x_p)_{p \in S}~:~d_{\Xp}(x_p, v_{0, p}) < \frac{2}{3}~\textrm{for all}~p \right\}.
\end{equation}
Here, by fundamental domain we mean an open set $U$ with compact closure such that $\Gamma_S \cdot U = \XS$. The action of $\Sunits$ on $\XS$ factors through $\PSunits$, and $R_S$ is also a fundamental domain for the action of $\PSunits$. Furthermore, the stabilizer of $v_0$ in $\PSunits$ is the image $\ProjHur$ of $\Hur^*$ in $\PSunits$, a finite subgroup of order 12 isomorphic to $\Afour$. We can then find a presentation for $\PSunits$ via the following theorem of Macbeath (\cite{Macbeath}, see also \cite[\S I.3]{Serre}).

%%%%%%%%%%%%%%%%%%%%
\begin{Theorem}\label{thm:Macbeath}
Let $X$ be a path connected and simply connected topological space, $G$ a group of homeomorphisms of $X$, and $U \subseteq X$ a path connected open set such that $G U = X$. Let
\[
\Sigma = \{ g \in G\ :\ U \cap g U \neq \emptyset \}
\]
and let $F(\Sigma)$ be the free group on $\Sigma$ with generators $x_\sigma$ for $\sigma \in \Sigma$. Then the homomorphism $F(\Sigma) \to G$ given by sending $x_\sigma$ to $\sigma$ is surjective with kernel equal to the normal subgroup of $F(\Sigma)$ generated by the words of the form $x_\sigma x_\tau x_{\sigma \tau}^{-1}$ for every pair $\sigma, \tau \in \Sigma$ such that
\[
U \cap \sigma U \cap \sigma \tau U \neq \emptyset.
\]
\end{Theorem}
%%%%%%%%%%%%%%%%%%%%

%%%%%%%%%%%%%%%%%%%%
\begin{Remark}\label{rem:MoreRelations}
Note that for any $\sigma$, $\tau \in \Sigma$, we can always add the relation $x_\sigma x_\tau = x_{\sigma\tau}$ to our presentation for $G$, and thus instead of checking triple intersections we can check for triples satisfying the weaker condition that $\sigma \tau \in \Sigma$ for $\sigma, \tau \in \Sigma$, at the price of adding some redundant relations.
\end{Remark}
%%%%%%%%%%%%%%%%%%%%

We then have the following algorithm. For simplicity, we consider $S$ not containing the prime $2$. Recall that $m_S = \prod_{p \in S} p$. In what follows, $\overline{X}$ denotes the image of a subset or element $X$ of $\Sunits$ in $\PSunits$.

%%%%%%%%%%%%%%%%%%%%
\begin{Algorithm}\label{alg:okay}\ 

\noindent
\textbf{Input:} A finite set $S$ of odd rational primes.

\noindent
\textbf{Output:} A presentation for $\PSunits$ with generators
\[
A = \left\{ a_{\overline{x}}\ :\ \overline{x} \in \overline{\Hur}\ \textrm{for}\ x \in \Hur\ \textrm{with}\ \Nm(x) \mid m_S \right\}
\]
and relations
\[
\Rels = \left\{ a_{\overline{\sigma}}\, a_{\overline{\tau}}\, {a_{\overline{\nu}}}^{-1} = 1~:~(a_{\overline{\sigma}},\, a_{\overline{\tau}},\, a_{\overline{\nu}}) \in Y \right\},
\]
where $Y$ is the set of triples $(a_{\overline{\sigma}}, a_{\overline{\tau}}, a_{\overline{\nu}}) \in A$ such that $\overline{\sigma}\, \overline{\tau} = \overline{\nu}$.
\end{Algorithm}
%%%%%%%%%%%%%%%%%%%%

%%%%%%%%%%%%%%%%%%%%
\begin{proof}
We need to show that the map $a_{\overline{x}} \mapsto \overline{x}$ is an isomorphism between the abstract group with generators $A$ and relations $\mathcal{R}$ and $\PSunits$. By Theorem \ref{thm:Macbeath} and Remark \ref{rem:MoreRelations}, it suffices to show that elements $a_{\overline{x}}$ of $A$ represent exactly those elements $\overline{x} \in \PSunits$ such that $\overline{x} R_S \cap R_S \neq \emptyset$, where $R_S$ is the fundamental domain for the action of $\PSunits$ defined in \eqref{eq:Rdomain}. Indeed, this means precisely that our generators are the generators given by Theorem \ref{thm:Macbeath}, and our relations are the relations given by Theorem \ref{thm:Macbeath} plus the possibly redundant relations considered in Remark \ref{rem:MoreRelations}.

If $a_{\overline{x}} \in A$ then for any $p \in S$, the associated element $x \in \Hur$ of reduced norm dividing $m_S$ has reduced norm in $p^\epsilon \Z_p^*$ for $\epsilon \in \{0, 1\}$ and hence either fixes $v_{0, p}$ (i.e., $\epsilon = 0$ and $x$ stabilizes the standard lattice $\mathcal{L}_{0, p}$ defined in \eqref{eq:standardL}) or maps $v_{0, p}$ to a neighbor in $\Xp$ (i.e., $\epsilon = 1$ and $x$ maps $\mathcal{L}_{0, p}$ to an index $p$ sublattice). When $\epsilon = 1$, notice that $x^{-1} v_{0, p}$ is also a neighbor of $v_{0, p}$, and thus the midpoint of the edge between $v_{0, p}$ and $x^{-1} v_{0, p}$ maps under $x$ to the midpoint of the edge between $v_{0, p}$ and $x\, v_{0, p}$. It follows that $\overline{x} R_p \cap R_p \neq \emptyset$ for each $p \in S$, with $R_p$ as in \eqref{eq:Rpdomain}. Thus $\overline{x} R_S \cap R_S \neq \emptyset$, and so $a_{\overline{x}}$ is one of the generators defined in Theorem \ref{thm:Macbeath}.

Conversely, suppose that $\gamma \in \Sunits$ has the property that $\gamma R_S \cap R_S \neq \emptyset$, so Theorem \ref{thm:Macbeath} says there should be a generator in $A$ associated with $\gamma$. It follows immediately from Remark \ref{rem:NeighborTransitive} that there exists some $x \in \Hur$ such that $\Nm(x) \mid m_S$ and $\overline{x} = \overline{\gamma}$. Therefore, $\gamma$ is associated with a generator $a_{\overline{x}} \in A$. Therefore $A$ is exactly the generating set from Theorem \ref{thm:Macbeath}, and we are done.
\end{proof}
%%%%%%%%%%%%%%%%%%%%

We also have the following improved algorithm, which produces more efficient presentations.

%%%%%%%%%%%%%%%%%%%%
\begin{Algorithm}\label{alg:better}\ 

\noindent
\textbf{Input:} A finite set $S$ of odd rational primes.

\noindent
\textbf{Output:} A presentation for $\PSunits$ with generators $A^\prime$ and relations $\Rels^\prime$.

The generators are of the form
\[
A^\prime = A_0 \cup \bigcup_{p \in S} A_p,
\]
where
\[
A_0 = \{ a_{\overline{x}}~:~\overline{x} \in \ProjHur,\ \overline{x} \neq 1 \},
\]
and $A_p$ consists of elements $a_{\overline{x}}$ for each vertex v in $N(v_{0,p})$, where we choose one $\overline{x} \in \overline{\Hur}$ such that $\overline{x}\, v_{0, p} = v$ for some $v \in N(v_{0, p})$ and $\overline{x}$ has a representative $x \in \Hur$ with $\Nm(x) = p$. Equivalently, $A_p$ consists of exactly one element for each of the $p+1$ orbits of the right action of $\mathcal{O}^*$ on the set elements of $\mathcal{O}$ of reduced norm $p$.

The relations $\Rels^\prime$ are all those of the following four types:
\begin{enumerate}

\item $a_{\overline{\sigma}} a_{\overline{\tau}} {a_{\overline{\nu}}}^{-1} = 1$ when $a_{\overline{\sigma}}, a_{\overline{\tau}}, a_{\overline{\nu}} \in A_0$ such that $\overline{\sigma}\, \overline{\tau} = \overline{\nu}$;

\item $a_{\overline{\sigma}} a_{\overline{\tau}} {a_{\overline{\nu}}}^{-1} = 1$ when $a_{\overline{\sigma}}, a_{\overline{\tau}} \in A_p$ for some $p \in S$ and $a_{\overline{\nu}} \in A_0$ with $\overline{\sigma}\, \overline{\tau} = \overline{\nu}$;

\item $a_{\overline{\sigma}} a_{\overline{\tau}} (a_{\overline{\nu}} a_{\overline{\alpha}} a_{\overline{\beta}})^{-1} = 1$ when $a_{\overline{\sigma}}, a_{\overline{\beta}} \in A_p$ and $a_{\overline{\tau}}, a_{\overline{\alpha}} \in A_q$ with $q < p$ and $a_{\overline{\nu}} \in A_0$ all satisfy $\overline{\sigma}\, \overline{\tau} = \overline{\nu}\, \overline{\alpha}\, \overline{\beta}$;

\item $a_{\overline{\nu}} a_{\overline{\sigma}} (a_{\overline{\tau}} a_{\overline{\mu}})^{-1} = 1$ when $a_{\overline{\sigma}}, a_{\overline{\tau}} \in A_p$ and $a_{\overline{\nu}}, a_{\overline{\mu}} \in A_0$ with $\overline{\nu}\, \overline{\sigma} = \overline{\tau}\, \overline{\mu}$.

\end{enumerate}
\end{Algorithm}
%%%%%%%%%%%%%%%%%%%%

%%%%%%%%%%%%%%%%%%%%
\begin{proof}
We first show that, considered as elements of $\PSunits$, the generators of Algorithm \ref{alg:okay} can be obtained from these generators.

Let $z \in \Hur$ be an element of reduced norm dividing $m_S$. Then $z$ maps $v_0 = (v_{0, p})_{p \in S}$ to a vertex
\[
v = (v_p)_{p \in S} \in N(v_0) = \prod_{p \in S} N(v_{0, p}).
\]
We claim there is an element $y = y_1 \cdots y_n \in \Hur$ with either $a_{\overline{y}_j} \in A_{p_j}$ ($p_1 < \cdots < p_n$) or $y_j = 1$, such that $y$ maps $v$ to $v_0$. First, note that each $y_i$ has reduced norm $p$, and hence fixes $y_{0, p_j}$ for $j \neq i$, but will permute its neighbors in $\mathfrak{X}_q$. Therefore, we take $y_i \in \Hur$ to be $1$ if $v_{p_i} = v_{0, p_i}$ or a representative in $\Hur$ of an element of $A_{p_i}$ that sends $y_{i + 1} \cdots y_n\, v_{p_i}$ to $v_{0, p_i}$ otherwise. Then $y_i \cdots y_n$ maps $v_{p_j}$ to $v_{0, p_j}$ for each $i \le i \le n$. The resulting element $y$ then maps $v$ to $v_0$, as claimed.

Then $z y$ is in the stabilizer of $v_0$, so the image of $z y$ in $\PSunits$ lies in $\ProjHur$. Therefore, there exists a unique $a_{\overline{x}} \in A_0 \sqcup \{1\}$ such that $\overline{z} = \overline{x}\, \overline{y}^{-1}$ in $\PSunits$, where $x \in \Hur$ represents $\overline{x} \in \ProjHur$. Since Algorithm \ref{alg:okay} shows that the $\overline{z}$ with $z \in \Hur$ of reduced norm dividing $m_S$ generate $\PSunits$, it follows that the elements $\overline{x}$ for $a_{\overline{x}} \in A^\prime$ also generate $\PSunits$.

The map $\psi$ from the group with generators $A^\prime$ and relations $\Rels^\prime$ to $\PSunits$ generated by $a_{\overline{x}} \mapsto \overline{x}$ is well defined since each element of $\Rels^\prime$ comes from an identity in $\PSunits$. Since we wrote any generator of $\PSunits$ coming from $A$ as a word in the $\psi(a_{\overline{x}})$ for $a_{\overline{x}} \in A^\prime$, $\psi$ is surjective. All that remains to be shown is that if $w$ is a word in the generators $A^\prime$ such that $\psi(w)$ is the identity, we can use the relations in $\Rels^\prime$ to reduce $w$ to the identity.

We first claim that, for every $a_{\overline{\sigma}} \in A_p$ and $a_{\overline{\nu}} \in A_0$, there exist $a_{\overline{\tau}} \in A_p$ and $a_{\overline{\mu}} \in A_0$ such that $\overline{\nu}\, \overline{\sigma} = \overline{\tau}\, \overline{\mu}$. In other words, we can move a generator from $A_p$ across a generator of type $A_0$ using relations in $\Rels^\prime$. The associated elements $\overline{\nu}\, \overline{\sigma} \in \overline{\Hur}$ fix $v_{0, q}$ for $p \ne q$ and send $v_{0, p}$ to a vertex adjacent to $v_{0, p}$. Thus there is a unique $a_{\overline{\tau}} \in A_p$ such that
\[
\overline{\nu}\, \overline{\sigma} \cdot v_{0, p} = \overline{\tau} \cdot v_{0, p}.
\]
Then $\overline{\tau}^{-1}$ also fixes $v_{0, q}$ for $q \ne p$, so $\overline{\tau}^{-1} \overline{\nu}\, \overline{\sigma}$ fixes $v_0$. Therefore
\[
\overline{\tau}^{-1} \overline{\nu}\, \overline{\sigma} = \overline{\mu}
\]
for some $\overline{\mu} \in \ProjHur$. This proves the claim.

Similarly, given $a_{\overline{\sigma}} \in A_p$ and $a_{\overline{\tau}}, \in A_q$ with $q < p$, we claim that there exist $a_{\overline{\beta}} \in A_p$, $a_{\overline{\alpha}} \in A_q$, and $a_{\overline{\nu}} \in A_0$ such that
\[
\overline{\sigma}\, \overline{\tau} = \overline{\nu}\, \overline{\alpha}\, \overline{\beta}.
\]
Indeed, $\overline{\sigma}\, \overline{\tau}$ fixes $v_{0, \ell}$ for $\ell \neq p, q$, and moves each of $v_{0, p}, v_{0, q}$ to a neighbor in its respective Bruhat--Tits tree. As in the proof of the previous claim, we can find $a_{\overline{\beta}} \in A_p$ and $a_{\overline{\alpha}} \in A_q$ so that
\[
\overline{\alpha}^{-1} \overline{\beta}^{-1} \overline{\sigma}\, \overline{\tau} \cdot v_0 = v_0,
\]
and is hence equal to some $\overline{\nu} \in \ProjHur$, which proves the claim. Note that a similar statement also holds for inverses of our generators.

By the above, we can use relations of type (3) to move a generator from $A_p$ or its inverse to the left across a generator from $A_q$ ($p < q$), at the expense of possibly introducing a generator of type $A_0$ to the left and, of course, possibly changing which element of $A_0$ or $A_p$ appears in the word. Similarly, relations of type (4) allow us to move a generator from $A_0$ or its inverse across a generator of type $A_p$ for some $p \in S$, again possibly changing which element of $A_0$ and $A_p$ appears. Applying these relations to consecutive positive or negative powers of generators appearing in the word, we can use the relations to assume that the word $w$ is of the form
\[
w = w_0 w_{p_1} \cdots w_{p_n}
\]
where $S = \{p_1, \dots, p_n\}$ ($p_1 < \cdots < p_n$) and $w_r$ is a (possibly empty) word in the generators from $A_r$, $r \in S \cup \{0\}$.

Since $\psi(w)$ is the identity, $\psi(w_p)$ must send $v_{0,p}$ to $v_{0,p}$ for each $p \in S$, since no $\psi(w_q)$ for $q \neq p$ moves $v_{0, p}$. Relations of type (2) allow us to replace $w_p$ with an element of $A_0$. Finally, using relations of type (1), we combine what remains into single generator corresponding to an element from $A_0$ that acts trivially on $\XS$ realized as an element of $\ProjHur$, and which therefore is the identity in the group with generators $A^\prime$ and relations $\Rels^\prime$. This completes the proof.
\end{proof}
%%%%%%%%%%%%%%%%%%%%

%%%%%%%%%%%%%%%%%%%%
\begin{Remark}
Using the results in \cite{Chinburg--Stover}, one can build explicit fundamental domains like those used in this section for the $S$-units of any definite quaternion algebra over $\Q$, provided $S$ is ``large enough'' as defined in $\cite{Chinburg--Stover}$. Typically the action will not be vertex transitive, but ideas analogous to the above will still give an algorithm to compute explicit presentations.
\end{Remark}
%%%%%%%%%%%%%%%%%%%%

%%%%%%%%%%%%%%%%%%%%
\section{Remarks on the congruence subgroup problem}\label{sec:csp}
%%%%%%%%%%%%%%%%%%%%

In this section, we use our presentations to conduct some experiments related to the congruence subgroup problem for $\Sunits$. We begin by briefly describing this problem, which is  open for the groups $\Sunits$ studied in this paper when $S$ contains at least two odd primes, i.e., when $\XS$ is a product of at least two trees. See \cite{Raghunathan} for more detailed surveys of the congruence subgroup problem.

A natural family of finite index subgroups of $\Sunits$ arise from the \emph{congruence subgroups}, which are defined as follows. Let $I$ be a nonzero two-sided ideal of $\HurS$. The \emph{congruence kernel of level} $I$, denoted $\Sunits(I)$, is subgroup of elements of $\Sunits$ congruent to the identity modulo $I$.

We can also define the congruence kernels as follows. For odd $p \notin S$, embed $\HamQ$ in $\M_2(\Qp)$. Since $\Hur$ is a maximal order of $\HamQ$, we can choose this embedding such that it induces an isomorphism of $\Hur \otimes_\Z \Zp$ with $\M_2(\Zp)$. It follows that $\Sunits$ embeds as a subgroup of $\GL_2(\Zp)$, from which we obtain a map
\[
\rho_{p^r} : \Sunits \to \GL_2(\Zp / p^r \Zp).
\]
The kernel of $\rho_{p^r}$ is precisely $\Sunits(I)$ for $I = \mathcal{P}^r$, where $\mathcal{P}$ is the prime of $\HurS$ with residue field $\Fp$, and we can extend this to arbitrary $I$ via the Chinese Remainder Theorem.

This leads to the following important question.

%%%%%%%%%%%%%%%%%%%%

\begin{CSP}
For every finite index subgroup $\Lambda$ of $\Sunits$, does there exists a nonzero two-sided ideal $I$ of $\HurS$ such that $\Sunits(I) < \Lambda$?
\end{CSP}

%%%%%%%%%%%%%%%%%%%%

If the answer to the above question is yes, then $\Sunits$ is said to have the \emph{congruence subgroup property} (CSP). Otherwise, we say that the CSP fails. When $S = \emptyset$, $\Sunits = \Hur^*$ is finite and one can find a proper ideal $I$ of $\Hur$ such that $\Hur^*(I) = \{1\}$. Therefore the trivial subgroup is a congruence kernel and $\Hur^*$ has the congruence subgroup property.

It is well known that $\Sunits$ does not have the CSP when $S$ contains exactly one odd prime. One way to see this is to observe that $\Sunits$ acts discretely and cocompactly on the tree $\XS$ in this case. It is shown in \cite[Part I]{Serre} that $\Sunits$ therefore contains a nonabelian free subgroup $F$ of finite index. Let $K$ be the kernel of a homomorphism of $F$ onto the alternating group $A_6$. We claim that $K$ is a finite index subgroup of $\Sunits$ that does not contain a congruence kernel. Indeed, one can show that $A_6$ is not isomorphic to a subquotient of
\begin{equation}\label{eq:CongruenceProduct}
\prod_{q \in \mathcal{Q}} \GL_2(\Z / q \Z)
\end{equation}
for any finite set $\mathcal{Q}$ of prime powers, and since every group $\Sunits / \Sunits(I)$ is a subgroup of some group of the form \eqref{eq:CongruenceProduct}, so $K$ cannot contain any $\Sunits(I)$.

%%%%%%%%%%%%%%%%%%%%

\begin{Remark}
Instead of $A_6$, we can choose one of the infinitely many finite groups that is not a quotient of a subgroup of some $\prod \GL_2(\Z / q \Z)$.
\end{Remark}

%%%%%%%%%%%%%%%%%%%%

Remarkably, the above basically sums up all our knowledge about the congruence subgroup problem for $S$-unit groups of $\HamQ$. There is a similar lack of knowledge for $S$-unit groups of arbitrary quaternion division algebras over number fields. The following exhausts the previous results that we know.

%%%%%%%%%%%%%%%%%%%%

\begin{Theorem}\label{thm:CSP}
Let $k$ be a number field and $B$ a $k$-quaternion algebra. For a maximal order $\Hur$ and a finite set of nonarchimedean places $S$ of $k$, let $\Sunits$ be the associated $S$-unit group.
\begin{enumerate}

\item If $\Sunits$ is finite, then it has the congruence subgroup property.

\item (Serre \cite{SerreCSP}) If $B \cong \M_2(k)$, then $\Sunits$ has the congruence subgroup property if and only if the ring of $S$-integers of $k$ contains a unit of infinite order.

\item (Lubotzky \cite{LubotzkyCSP}) If $B$ is a division algebra and $\Sunits$ is a lattice in $\GL_2(K)$ for any local field $K$ of characteristic zero, then $\Sunits$ does not have the congruence subgroup property.

\end{enumerate}
\end{Theorem}

%%%%%%%%%%%%%%%%%%%%

For $B$ and $\Sunits$ as in Theorem \ref{thm:CSP}, we know of no further results about the congruence subgroup problem. In particular, as far as we know, for $\HamQ$ the congruence subgroup problem for $\Sunits$ is open in all cases except when $S$ contains at most one odd prime. It is sometimes called Serre's Conjecture that when $S$ contains at least two odd primes, $\Sunits$ and $\PSunits$ should have the congruence subgroup property. The remainder of this section collects some data related to this important open question.

Using MAGMA \cite{MAGMA}, we computed finite index subgroups of $\PSunits$ for various $S$ containing at least two primes. All of our observations support the congruence subgroup property holding when $S$ contains at least two odd primes. We considered all finite quotients of $\PSunits$ of order at most $n$ for some small $S$. In the table below we tabulate all the groups which occur in composition series for such quotients for the $S$ and $n$ indicated.
\begin{table}[h]
\begin{center}
\begin{tabular}{|c|c|c|c|}
\hline
$S$ & $n$ & $\Z / m \Z$ & $\PSL_2(\F_p)$ \\
\hline
$\{3, 5\}$ & $30{,}000$ & $2, 3$ &$7, 11, 13, 17, 19, 23, 29, 31$ \\
\hline
$\{5, 11\}$ & $20{,}000$ & $2, 3$ & $7, 13, 17, 19, 23$ \\
\hline
$\{37, 43\}$ & $15{,}000$ & $2, 3, 5$ & $5, 7, 11, 13, 17, 19, 23$ \\
\hline
$\{3, 5, 7\}$ & $20{,}000$ & $2, 3$ & $11, 13, 17, 19, 23$ \\
\hline
$\{3, 7, 11\}$ & $20{,}000$ & $2, 3, 5$ & $5, 13, 17, 19, 23$ \\
\hline
$\{7, 11, 13\}$ & $15{,}000$ & $2, 3, 5$ & $5, 17, 19, 23$ \\
\hline
$\{11, 13, 17, 19\}$ & $10{,}000$ & $2, 3$ & $5, 7$ \\
\hline
$\{3, 5, 7, 11, 13\}$ & $10{,}000$ & $2, 3$ & $17, 19$ \\
\hline
\end{tabular}
\caption{Groups appearing in the composition series for finite quotients of $\PSunits$ of order $\le n$.}
\end{center}
\end{table}

We now briefly describe why our experiments are consistent with the congruence subgroup property, even though many quotients are not exactly of the form $\PSL_2(\Z / N \Z)$. Cyclic composition factors arise from either the congruence subgroups of the unique maximal order in $\Hur_2$, quotients of the binary tetrahedral group, or the soluble part of $\PSL_2(\Z / p^n \Z)$. Finally, note in particular that there are no homomorphisms onto $\mathrm{A}_6$.

%%%%%%%%%%%%%%%%%%%%

\section{Presentations}\label{sec:presentations}

%%%%%%%%%%%%%%%%%%%%

We now tabulate presentations for $\PSunits$ for several small values of $S$. Since we are considering the projective units,we give an integral representative for each generator, that is, an element of $\Z[1, I, J, I J]$. To our knowledge, these are the first known presentations of $S$-arithmetic lattices in a semisimple Lie group where $S$ contains at least two places at which the associated algebraic group is isotropic (i.e., not compact). Unpublished work of the late Fritz Grunewald presented some groups $\Sunits^1$, the subgroup of elements of reduced norm $1$, for many small $S$ contained in this paper, but unfortunately we were not able to obtain a copy of this work.

%%%%%%%%%%%%%%%%%%%%

\subsubsection*{S = \{3, 5\}}

\noindent
Generators:
\begin{align*}
a =& -1 + I - J - 3 I J \\
b =& -9 - 7 I - J + 7 I J
\end{align*}

\noindent
Relators:
\begin{align*}
r_1 =& (b^{-1} a^{-1} b a^{-1})^3 \\
r_2 =& (b^{-1} a^{-2} b a^{-1} b^{-1} a^{-1})^2 \\
r_3 =& (a^{-1} b^{-1} a^{-1} b^{-1} a^{-1} b a^{-1})^2 \\
r_4 =& b^{-1} a b a b^{-1} a^{-1} b^2 a b^{-1} a b a^2 b^{-1} a b a b^{-1} a^2 b a^2 b^{-1} a^{-1} b a^{-2} b^{-1} a^{-2} \\
r_5 =& (b a^2 b^{-1} a b a^{-1} b)^2 \\
r_6 =& b^{-1} a^3 b a^2 b^{-1} a b^{-1} a^{-2} b a^{-1} b^{-1} a \\
r_7 =& b^{-2} a^{-1} b a^{-1} b^{-1} a b a^2 b^{-2} a^{-2} b a^{-1} \\
r_8 =& a b^{-1} a^2 b a^{-1} b^{-1} a^{-2} b a^{-2} b^{-1} a b a
\end{align*}

%%%%%%%%%%%%%%%%%%%%

\subsubsection*{S = \{3, 7\}}

\noindent
Generators:
\begin{align*}
a =& -1 + I - J - 3 I J \\
b =& -1 - I - J - 5 I J
\end{align*}

\noindent
Relators:
\begin{align*}
r_1 =& b a b a^{-2} b a b^{-1} a^{-1} b^{-1} a^2 b^{-1} a^{-1} \\
r_2 =& a^3 b a^{-2} b a b a^2 b^{-1} a^{-1} b^{-3} a^{-1} b^{-1} \\
r_3 =& b a b^{-1} a^{-1} b^{-1} a^{-1} b a b^2 a b^2 a^{-2} b a b \\
r_4 =& (a^2 b^{-1} a^{-1} b^{-2} a^{-1} b^{-1} a)^2 \\
r_5 =& a b^3 a b^3 a b a^{-2} b a^3 b a^{-2} b \\
r_6 =& b^{-2} a b^2 a b a^{-2} b^3 a b a^{-2} b^2 a^2 b^{-1} a^{-1} b^{-2} a^{-1} b^{-2} a^{-1}
\end{align*}

%%%%%%%%%%%%%%%%%%%%

\subsubsection*{S = \{3, 11\}}

\noindent
Generators:
\begin{align*}
a =& 1 + I - J - I J \\
b =& -1 - I - J - 3 I J \\
c =& -1 + I- 3 I J
\end{align*}

\noindent
Relators:
\begin{align*}
r_1 =& a^3 \\
r_2 =& (b^{-1} c a^{-1})^2 \\
r_3 =& (b, a^{-1})^2 \\
r_4 =& (c^{-1} b a^{-1} b)^2 \\
r_5 =& b a^{-1} b^-2 a c^{-1} a b^{-1} c^{-1} \\
r_6 =& c^{-1} a b^{-1} c^{-1} a c b^{-1} a^{-1} c^{-1} \\
r_7 =& (b^2 a^{-1} b^{-1} a^{-1})^2 \\
r_8 =& (b a b^{-1} a^{-1} c^{-1})^2
\end{align*}

%%%%%%%%%%%%%%%%%%%%

\subsubsection*{S = \{5, 7\}}

\noindent
Generators:
\begin{align*}
a =& 1 - I + J - I J \\
b =& -J - 2 I J \\
c =& -1 + I + J - 5 I J
\end{align*}

\noindent
Relators:
\begin{align*}
r_1 =& b^2 \\
r_2 =& a^3 \\
r_3 =& (c^{-1} a b)^2 \\
r_4 =& (a, c^{-1})^2 \\
r_5 =& (b c a^{-1} c^{-1} a)^2 \\
r_6 =& (c a^{-1} c^{-1} a^{-1})^3 \\
r_7 =& c a c^{-1} a b c a c^{-1} a^{-1} c^{-1} a^{-1} b a c a \\
r_8 =& c^{-1} a^{-1} b a c^2 a c^{-1} a^{-1} b a c a c^{-1} a
\end{align*}

%%%%%%%%%%%%%%%%%%%%

\subsubsection*{S = \{3, 5, 7\}}

\noindent
Generators:
\begin{align*}
a =& 1 + I - J - I J \\
b =& -1 - I - J - 3 I J \\
c =& -I - 2 I J \\
d =& -1 - I - J - 5 I J
\end{align*}

\noindent
Relators:
\begin{align*}
r_1 =& c^2 \\
r_2 =& a^3 \\
r_3 =& b^{-1} d a d^{-1} b a^{-1} \\
r_4 =& b d c d^{-1} b^{-1} c \\
r_5 =& c a^{-1} d^{-1} c a d \\
r_6 =& (d^{-1}, a)^2 \\
r_7 =& (d b a^{-1} d)^2 \\
r_8 =& (c a^{-1} b^2)^2 \\
r_9 =& (d a b a^{-1})^2 \\
r_{10} =& b^{-1} d c a d^{-1} b^{-1} a c a^{-1} \\
r_{11} =& c a^{-1} b^{-1} a^{-1} c d a^{-1} d^{-1} b^{-1} a \\
r_{12} =& b a d a d^{-1} a b^2 a b^{-1} a \\
r_{13} =& d^{-1} b^{-1} a^{-1} b d^{-1} b^2 a d a d^{-1} \\
r_{14} =& (a^{-1} d a^{-1} d^{-1})^3 \\
r_{15} =& d^2 a d^{-1} a b d^{-1} a^{-1} d a^{-1} d^{-1} b^{-1} \\
r_{16} =& d^{-1} a^{-1} b^{-1} a c b^{-1} a d^{-1} a c a^{-1} d a^{-1} d^{-1} \\
r_{17} =& c d a^{-1} d^{-1} a^{-1} d^{-1} a^{-1} b^{-1} a c b^{-1} d^{-1} a d a d^{-1} \\
r_{18} =& (d a^{-1} d a^{-1} d^{-1} a^{-1} c a^{-1})^2
\end{align*}

%%%%%%%%%%%%%%%%%%%%

\subsubsection*{S = \{3, 5, 11\}}

\noindent
Generators:
\begin{align*}
a =& -1 - I - J - 3 I J \\
b =& -1 - 2 I J \\
c =& -1 + J - 3 I J
\end{align*}

\noindent
Relators:
\begin{align*}
r_1 =& b^2 c b a^{-1} c^{-2} b^{-1} a c^{-1} \\
r_2 =& (b^{-1} c^{-1} b^{-1} a c^{-1} a^{-1})^2 \\
r_3 =& b c b a b a^{-1} b^{-1} c^{-1} b^{-1} a c^{-1} b^{-1} c a^{-1} \\
r_4 =& b a c^{-1} b^{-1} a c^{-1} a^2 b^{-1} c^{-1} b^{-1} a c^{-1} a \\
r_5 =& b^{-2} a^{-1} b^{-1} c^{-1} b^{-1} c^{-1} b^{-1} a c^{-1} a b a^{-1} b^{-1} c^{-1} b^{-1} \\
r_6 =& b a c^{-1} b^{-1} a c^{-1} a b^{-1} c a^{-1} c a^{-1} b c a^{-1} c^{-1} \\
r_7 =& c a^{-1} b c b^{-1} a^{-1} b^{-1} c b^{-1} c^{-1} b^{-1} a c^{-1} a b c \\
r_8 =& c b^2 a c^{-1} b^{-1} a c^{-1} a c a^{-1} b c b a^{-1} c a^{-1} b \\
r_9 =& a^{-1} c^{-1} a^{-1} c a^{-1} b c b^{-1} a c^{-1} b a b^{-1} c^{-1} b^{-2} a c \\
r_{10} =& c^2 a^{-1} b c b a^{-2} c^{-2} b^{-1} a^2 b^{-1} c^{-1} b^{-1} a b \\
r_{11} =& c^{-1} a^{-1} b^2 c b a^{-1} b^{-1} a^{-1} b^{-1} a b^2 a c^{-1} b^{-1} a c^{-1} \\
r_{12} =& b c a c^{-1} b^{-1} a c^{-1} a c^{-2} b^{-1} a^{-1} c a^{-1} b c a^{-1} c \\
r_{13} =& (b^2 c a^{-1} b c b^2 a)^2 \\
r_{14} =& a^{-2} c a^{-2} c a^{-1} b c a^{-1} c a^{-2} b a b c^{-1} b^{-1} a c^{-1} \\
r_{15} =& b a b^2 a c^{-1} b^{-2} a b^{-1} a^{-1} b^{-2} a^{-3} b^{-1} c^{-1} b^{-1} a \\
r_{16} =& a c b c b a^{-1} c^{-1} b^{-1} c^{-1} b^{-1} a c^{-1} b c b a b^{-1} a^{-1} b^{-2} a^{-1} b^{-1} \\
r_{17} =& b a b^2 a b a c^{-1} a b a b c^{-1} b^{-1} a c^{-1} b a^{-1} b c b a \\
r_{18} =& b^2 c a^{-1} b c a^{-1} c b a^{-1} c^{-1} a^{-1} c a^{-1} b c b a^{-1} b a^{-1} b a \\
r_{19} =& b^{-2} c^{-1} b^{-1} a c^{-1} a b^{-1} a b^{-1} a^{-1} b^{-2} a^{-1} b c b^{-1} a^{-1} c a^{-1} b^{-1} a \\
r_{20} =& (b a c^{-1} a^{-1} c a^{-1} b c b^2 a)^2 \\
r_{21} =& (c^{-1} a^{-2} c a^{-1} b c b c^{-1} b^{-1} a c^{-1} a^{-1})^2 \\
r_{22} =& a^{-1} b c b a c^{-1} b^{-1} a c^{-1} a^3 c a^{-1} b c b^{-1} a^{-1} \\
&b^{-1} c a^{-1} b c a^{-1} b^{-1} c b^{-2} c^{-1} b^{-1} a c^{-1} b^{-2} a^{-1} c b c
\end{align*}

%%%%%%%%%%%%%%%%%%%%

\section{Sage code}\label{sec:code}

%%%%%%%%%%%%%%%%%%%%

Below is the Sage \cite{Sage} code we used to compute our presentations via Algorithm \ref{alg:better}. Note that our code uses Magma during the Sage routine.

%%%%%%%%%%%%%%%%%%%%

\begin{scriptsize}
\begin{verbatim}
H.<i,j,k>=QuaternionAlgebra(QQ,-1,-1)

# On input an integer n, outputs all elements of reduced norm n in the 
# Hurwitz order O
def findofnorms(n): 
  m = 4*n
  U = []
  for a4 in xrange(-sqrt(m),sqrt(m)+1):
    for a3 in xrange(-sqrt(m-a4^2),sqrt(m-a4^2)+1):
      for a2 in xrange(-sqrt(m-a4^2-a3^2),sqrt(m-a4^2-a3^2)+1):
        for a1 in xrange(-sqrt(m-a4^2-a3^2-a2^2),sqrt(m-a4^2-a3^2-a2^2)+1):
          if m-a4^2-a3^2-a2^2-a1^2 == 0:
            U = U+[H([1/2*a1, 1/2*a2, 1/2*a3, 1/2*a4])] # element is in O
  return U

  
# Find elements of reduced norm n, when n is prime, in the Hurwitz order O up 
# to left O^* orbit. If n is one, returns O^*/<+-1>.
def findofnormso(n):
  U = findofnorms(n)
  t = len(U)
  Ostar = findofnorms(1) # O^*
  if n == 1:
    return [r for r in U if r>-r]
  if n == 2:
    n = 0
  for i in xrange(n+1):
    b = 1/U[i]
    j = i+1
    while j < t:
      if U[j]*b in Ostar:
        U.pop(j)
        t = t-1
      else:
        j += 1
  return U

  
# returns max(a,-a)      
def bigger(a):
    return a if a >- a else -a


# On input a list a=S of different primes, finds a presentation for 
# overline{O_S}^* and writes it to magmastring.c, to be read by Magma 
def findpresentation(a):
  a = [1]+a
  si = []
  r = [0]
  dummy = 0
  counter = [0,0,0]
  for i in xrange(0,len(a)):
    b = findofnormso(a[i])
    r.append(r[i]+len(b))
    si.append(b)
  str = 'F<[t]>:=FreeGroup('+repr(r[len(a)])+');\nFP:=quo<F|['

  # p*p relations (includes O^* O^*): Type 1,2
  for s in xrange(len(a)):
    for u in xrange(r[s+1]-r[s]):
      for v in xrange(r[s+1]-r[s]):
        try:
          l=si[0].index(bigger(si[s][u]*si[s][v]/a[s]))
          counter[0] += 1
          if dummy == 0:
            str += 't['+repr(r[s]+u+1)+']*t['+repr(r[s]+v+1)+']=t['+repr(l+1)+']'
            dummy = 1
          else:
            str += ',t['+repr(r[s]+u+1)+']*t['+repr(r[s]+v+1)+']=t['+repr(l+1)+']'
        except ValueError:
          pass

  # p*p' relations: Type 3
  for t in xrange(1,len(a)):
    for s in xrange(t+1,len(a)):
      ind = []
      li = []
      for u1 in xrange(r[t+1]-r[t]):
        for u2 in xrange(r[s+1]-r[s]):
          for d in xrange(r[1]-r[0]):
            ind.append(bigger(si[0][d]*si[t][u1]*si[s][u2]))
            li.append([d,u1,u2])
      for u1 in xrange(r[t+1]-r[t]):
        for u2 in xrange(r[s+1]-r[s]):       
          l = ind.index(bigger(si[s][u2]*si[t][u1]))        
          str += ',t['+repr(r[s]+u2+1)+']*t['+repr(r[t]+u1+1)
          str += ']=t['+repr(r[0]+li[l][0]+1)+']*t['
          str += repr(r[t]+li[l][1]+1)+']*t['+repr(r[s]+li[l][2]+1)+']'
          counter[1] += 1
  
  # p*d=d*p: Type 4
  for s in xrange(1,len(a)):
    ind = []
    li = []
    for u in xrange(r[s+1]-r[s]):
      for d in xrange(r[1]-r[0]):
        ind.append(bigger(si[s][u]*si[0][d]))
        li.append([u,d])
    for u in xrange(r[s+1]-r[s]):
      for d in xrange(r[1]-r[0]):
        l = ind.index(bigger(si[0][d]*si[s][u]))
        str += ',t['+repr(r[0]+d+1)+']*t['+repr(r[s]+u+1)+']=t['
        str += repr(r[s]+li[l][0]+1)+']*t['+repr(r[0]+li[l][1]+1)+']'
        counter[2] += 1

  str = str+']>;\n'+'R:=ReduceGenerators(FP);\n'+'print R;'
  f = open('magmastring2.c','w')
  f.write(str)
  f.close()
  print 'Type 1+2:', counter[0],'\nType 3:', counter[1],'\nType 4:', counter[2]

# On input a list n of different primes S, finds a short presentation for 
# overline{O_S}^* and appends it in results.txt, uses Magma
def findpresentation2(n):
  g = open('results.txt','a')
  findpresentation(n)
  g.write('Results on input '+repr(n)+':\n')
  g.write(magma.load('magmastring2.c'))
  g.write('\n\n')
  g.close()
\end{verbatim}
\end{scriptsize}

%%%%%%%%%%%%%%%%%%%%

\end{document}